\documentclass[10pt]{article}

\usepackage{hyperref}
\usepackage{xcolor}
\usepackage{amssymb}
\usepackage{amsfonts}
\usepackage{amsthm}
\usepackage{amsmath}
\usepackage{float}
\usepackage{bigints}
\usepackage{fullpage}
\usepackage{mathtools}
\usepackage[utf8]{inputenc} 
\usepackage{cancel}
\usepackage{verbatim}
\usepackage{soul}
\usepackage{color}
\usepackage{tikz}
\usetikzlibrary{calc}
\usepackage{caption}
\usepackage{subcaption}
\usepackage{pgfplots}
\usepackage{ulem} 
\usepackage[inkscapeformat=png]{svg} 
 
\pgfplotsset{compat = newest}

\newtheorem{theorem}{Theorem}

\newtheorem{lemma}[theorem]{Lemma}

\newtheorem{definition}{Definition}
\newtheorem{problem}[theorem]{Problem}
\newtheorem{corollary}[theorem]{Corollary}

\newtheorem{claim}[theorem]{Claim}
\newcommand{\inter}{\rm int}

\newcommand{\conv}{\rm conv}

\newcommand{\oo}{\mathbf{o}}

\newcommand{\cl}{{\rm cl}}

\newcommand{\noshow}[1]{}

\DeclareRobustCommand{\hlc}[2][yellow]{ {\sethlcolor{#1} \hl{#2}} }

\title{Vertex Classification of Planar C-Polygons
\footnote{}
\author{Illya Ivanov and Cameron Strachan}}

\date{}

\begin{document}

\maketitle

\begin{abstract}
\noindent Given a convex domain $C$, a $C$-polygon is an intersection of $n\geq 2$ homothets of $C$. If the homothets are translates of $C$ then we call the intersection a translative $C$-polygon. This paper proves that if $C$ is a strictly convex domain with $m$ singular boundary points, then the number of singular boundary points a $C$-polygon has is between $n$ and $2(n-1)+m$. For a translative $C$-polygon we show the number of singular boundary points is between $n$ and $n+m$.
\end{abstract}

\section{Introduction}\label{S1}

Let $\mathbb{E}^2$ denote the 2-dimensional planar Euclidean space, and $\mathbb{S}^1$ the unit sphere inside $\mathbb{E}^2$. A set $X\subseteq\mathbb{E}^2$ is convex when the line segment connecting any two points in $X$ is contained in $X$. We denote the interior, boundary and convex hull of $X$, as $\text{int}(X)$, $ \mathrm{bd}(X)$ and $\conv(X)$ respectively. A set $C\subseteq \mathbb{E}^2$ is convex domain when it is compact, convex, and has non-empty interior. If the relative interior of the line segment connecting any two boundary points is contained in the interior of $C$, then we call $C$ a strictly convex domain. 

Let $C$ be a convex domain, the Gauss mapping of $C$, which is denoted as $\Gamma_C$, is the mapping between the boundary of $C$ and $\mathbb{S}^1$ where each boundary point, $x$, gets mapped to the set of normal unit vectors of the supporting lines of $C$ at the boundary point $x$. The Gauss mapping of any convex domain is always surjective and functional. When the Gauss mapping of $C$ is well defined we call $C$ smooth, and when the Gauss mapping is injective this is equivalent to $C$ being strictly convex. When a convex domain is both strictly convex and smooth we call it a smooth strictly convex domain, which is equivalent to having a bijective function as a Gauss mapping. Let $X_i$ be a subset of $\mathbb{E}^2$ for each $i\in \{1, 2, \dots, n\}$, we call $\cap_{i=1}^n X_i$ a reduced intersection when $\cap_{i=1}^n X_i\subsetneq \cap_{i=1,i\neq j}^n X_i$ for every $j\in \{1,2,\dots,n\}$. We also denote $X_1+X_2=\{x_1+x_2 | x_1\in X_1,x_2\in X_2\}$ and $\lambda X_1=\{\lambda x_1 |x_1\in X_1\}$ where $\lambda$ is a positive real.

\begin{definition} \label{D1}
We call a finite intersection of sets, $\cap_{i=1}^n X_i$ a proper intersection when $\cap_{i=1}^nX_i$ has non-empty interior and is a reduced intersection. We call $\cap_{i=1}^nX_i$ an improper intersection otherwise.
\end{definition}

The main purpose of this paper is to provide a somewhat analogous result to the famous upper bound theorem of McMullen \cite{UB}, in the planar case. Given a convex polytope in $d$-dimensional space, made by intersecting $n$ halfspaces, the upper bound theorem provides sharp upper bounds for the number of $i$-dimensional faces this polytope can have for $i\in\{0,1,\dots,d-1\}$. We investigate the same question one could ask about $C$-polygons which are defined as follows. It should be noted that the results and definitions in this paper are strictly planar, although generalization to higher dimensions can naturally be defined.

\begin{definition} \label{D2}
Given a convex domain, $C\subseteq \mathbb{E}^2$, a set of $n$ points $\{x_1,x_2,\dots, x_n\}\subseteq \mathbb{E}^2$, and a set of $n$ positive scalars $\{\lambda_1,\lambda_2,\dots,\lambda_n\}\subseteq \mathbb{R}$ where $n\geq 2$, we denote $H_i=x_i+\lambda_iC$ and  $T_i=x_i+C$, and define a $C$-polygon, $H$, and a translative $C$-polygon, $T$, as the following intersections if these intersections are proper.
$$H:=\cap_{i=1}^nH_i=\cap_{i=1}^nx_i+\lambda_iC~~\text{ and }~~T:=\cap_{i=1}^nT_i=\cap_{i=1}^nx_i+C $$ 
We call each $H_i$ or $T_i$ for $i\in \{1,2,\dots n\}$ a generating homothet or generating translate respectively.
\end{definition}

When $C$ is a ball, then the definitions of $C$-polygons and translative $C$-polygons, define classical notions of generalized ball polygons and ball polygons \cite{BP}, respectively. We investigate a somewhat dual version of \cite{SOTU} and study the complexity of the boundary structure of $C$-polygons and translative $C$-polygons, and how varied the boundary structure can be if we fix the number of generating translates or homothets. Specifically, we give upper and lower bounds on how many singular boundary points these objects can have when $C$ is a strictly convex domain. We call these singular boundary points the vertices of the $C$-polygon and denote the set of them as $Vert(H)$. The following is the main result established in this paper.

\begin{theorem}\label{T1}
Let $T=\cap_{i=1}^nT_i$ and $H=\cap_{i=1}^nH_i$ be a translative $C$-polygon and $C$-polygon respectively. If $C$ is a strictly convex domain with $m$ singular boundary points then
$$ n\leq |Vert(T)|\leq n+m, \text{ and  } n\leq |Vert(H)|\leq 2(n-1)+m.  $$
\end{theorem}

Our general approach in proving Theorem \ref{T1} is to handle the case where $m=0$, or equivalently the case where $C$ is smooth, and then extend this to the case where $m$ is positive. In section \ref{S2}, we establish some basic properties of $C$-polygons where $C$ is a smooth strictly convex domain. In section \ref{S3} we prove Theorem \ref{T1} when $m=0$, and in section \ref{S4} we prove the case of Theorem \ref{T1} where $m$ is positive. Finally in section \ref{S5} we explore the sharpness of the bounds established in Theorem \ref{T1}. 

\section{Basic Properties for \textit{C}-polygons where \textit{C} is a smooth strictly convex domain}\label{S2}

We start with defining a natural face structure for $C$-polygons, where $C$ is a smooth strictly convex domain. It suffices to define them for $C$-polygons as translative $C$-polygons are a particular example of $C$-polygons. Let $H=\cap_{i=1}^nH_i$ be a $C$-polygon, each generating homothet will be a smooth strictly convex domain, which implies $H$ must also be a strictly convex domain. Since $H$ is a strictly convex domain, its boundary can be separated into singular points and smooth boundary arcs between singular points, we define these to be vertices and edges respectively of our $C$-polygon. However, in order to prove our result we come up with a more useful equivalent characterization of vertices and edges. We can do this because $H$ is not merely a strictly convex domain, it is a proper intersection of $n$ homothets of $C$, which allows for the equivalence of the following definitions with vertices and edges described above.
\begin{definition} \label{D3}
Given a generating homothet $H_j$ of a $C$-polygon $H$, the edge family of $H_j$ is denoted and defined as $$\mathcal{E}_j= (\mathrm{bd} (H)\cap  \mathrm{bd} (H_j))\setminus A .$$
The intersection $\mathrm{bd} (H)\cap  \mathrm{bd} (H_j)$ is a union of maximally connected closed boundary arcs of $H_j$ and the set $A$ contains all the arcs that are singleton points. We call each maximally connected closed boundary arc of $\mathcal{E}_j$ an edge of $H$. We further denote $E_j$ as the set whose elements are all the edges in $\mathcal{E}_j$.
\end{definition}
\begin{definition} \label{D4}
A point $v\in  \mathrm{bd}(H)$ is a vertex if it lies in the boundary of at least two generating homothets.
\end{definition}

\begin{figure}[!ht]
\centering
\begin{subfigure}{.33\textwidth}
  \centering
  \begin{tikzpicture} 
\node [above right, inner sep=0] (image) at (0,0)   {\includegraphics[width=0.8\linewidth]{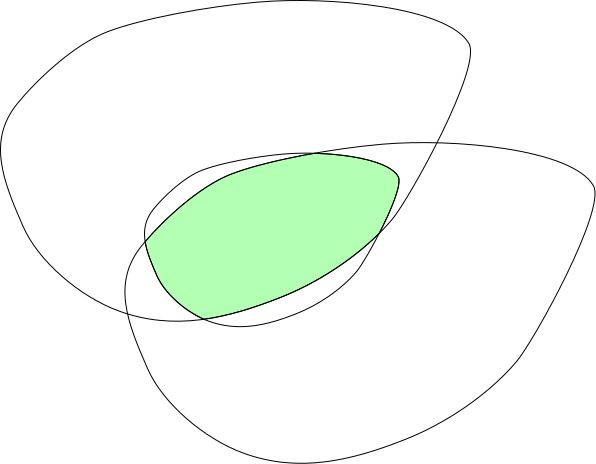}};   
\begin{scope}[
x={($0.1*(image.south east)$)},
y={($0.1*(image.north west)$)}];
\end{scope}
\end{tikzpicture}
\caption{}
\end{subfigure}%
\begin{subfigure}{.33\textwidth}
  \centering
    \begin{tikzpicture} 
\node [above right, inner sep=0] (image) at (0,0) {\includegraphics[width=0.8\linewidth]{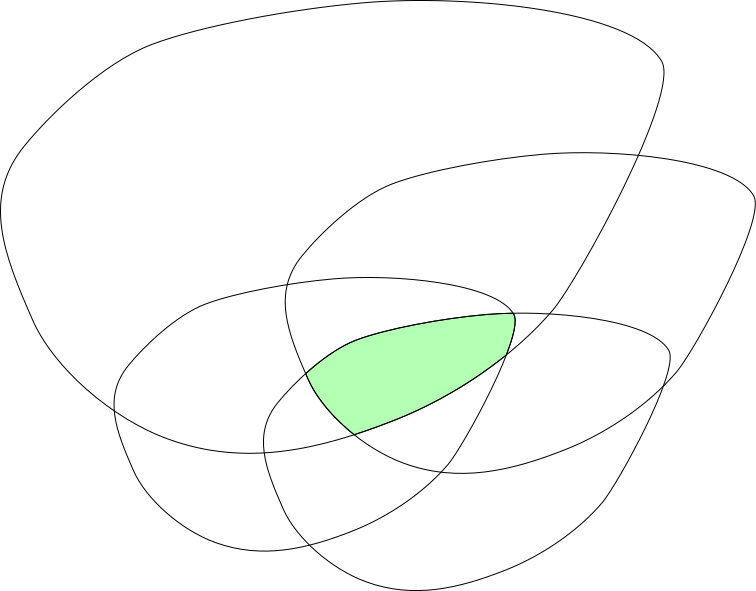}};  
\begin{scope}[
x={($0.1*(image.south east)$)},
y={($0.1*(image.north west)$)}];
       
    
\end{scope}
\end{tikzpicture}
\caption{}
\end{subfigure}
\begin{subfigure}{.33\textwidth}
  \centering
    \begin{tikzpicture} 
\node [above right, inner sep=0] (image) at (0,0) 
{\includegraphics[width=0.8\linewidth]{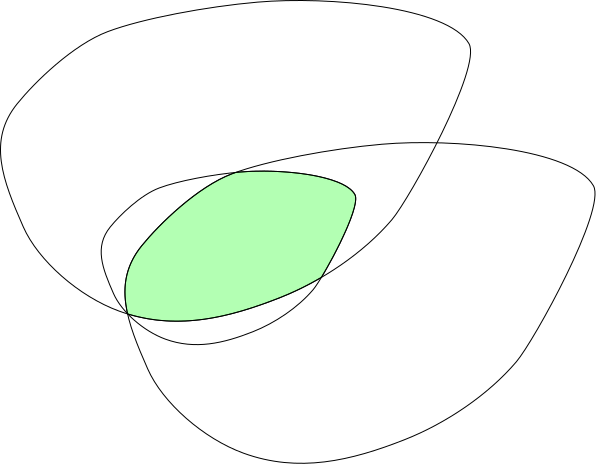}};
\begin{scope}[
x={($0.1*(image.south east)$)},
y={($0.1*(image.north west)$)}];
       
    
\end{scope}
\end{tikzpicture}
\caption{}
\end{subfigure}
\caption{Three examples of $C$-polygons.}
\label{F1}
\end{figure}

Diagram (c) in Figure \ref{F1} shows that there can be cases in which the non-singleton maximally connected component of Definition \ref{D3} comes into play. If an edge family has more than one edge in it we call it a multi-edge family, and if it has only one edge we call it a singleton-edge family.
In order to show these definitions of vertices and edges align with singular points and smooth boundary arcs of $H$ between singular boundary points, we need the following lemma.

\begin{lemma} \label{PL}
Suppose $H_1$ and $H_2$ are two homothets of a strictly convex domain, $C$, whose intersection is proper, then $$| \mathrm{bd}(H_1)\cap  \mathrm{bd}(H_2)|=2.  $$
\end{lemma}
\begin{proof} Let $H_1 = x_1 + \lambda_1 C$ and $H_2 = x_2 + \lambda_2C$, first we consider the case when $\lambda_1 \neq \lambda_2$. In this case we can, without loss of generality, choose an origin, $\oo\in \mathbb{E}^2$, such that $H_2 = \lambda H_1$ for some $\lambda > 1$. Since $H_1\cap H_2$ is a proper intersection, $\oo$ must be outside $H_1$, and $ \mathrm{bd}(H_1) \cap  \mathrm{bd}(H_2)$ must contain at least two points. If the intersection of the boundary of the two homothets has cardinality greater than two, then we have distinct $q,r,s\in  \mathrm{bd}(H_1)\cap  \mathrm{bd}(H_2)$; notice that these points can't be colinear as they are boundary points of a strictly convex domain. Naturally, $\lambda q,\lambda r, \lambda s\in \mathrm{bd}(\lambda H_1) = \mathrm{bd}(H_2)$.

Since $\oo \notin H_1$, one of the points $q, r, s$ must belong to the positive hull of the other two. Without loss of generality, suppose $s = \mu_1 q + \mu_2 r$ where $\mu_1, \mu_2 \geq 0$. In fact $\mu_1,\mu_2 >0$, suppose for example $\mu_2=0$, then we would have the points $q, \mu_1q, \lambda q, \lambda \mu_1 q$ all lie on a straight line and also in $\mathrm{bd} (H_2)$, a contradiction since at least three out of these four points must be distinct and $H_1\cap H_2$ is strictly convex. Also note that $\mu_1 + \mu_2 \neq 1$, since $q,r,s$ aren't colinear. Then one can see that if $\mu_1 + \mu_2 < 1$, then $\lambda s \in  \text{int}(\conv \left\{q,r,s,\lambda q, \lambda r \right\})$ and cannot lie in $\mathrm{bd}(H_2)$. If in turn $\mu_1 + \mu_2 > 1$, then $s$ lies in the interior of the convex hull of the other five points, and cannot be in $\mathrm{bd}( H_2)$.

For the other case suppose $\lambda_1 = \lambda_2$, and let $H_2 = H_1 + x$. 
Let $y$ be a normal vector to $x$. The inner product values $\langle q,y \rangle, \langle r, y \rangle, \langle s, y \rangle$ are all distinct, otherwise four of the points $q,r,s,q+x,r+x,s+x$ would lie on a same line and in the boundary of $H_1$. Without loss of generality, suppose $\langle q,y \rangle<\langle r, y \rangle<\langle s, y \rangle$. Then $r = \mu_1 q + \mu_2 s + k x$ for some $k \neq 0$ and some $\mu_1, \mu_2 \in (0,1)$ such that $\mu_1 + \mu_2 =1$. If $k < 0$, then $r + x \in  \text{int}(\conv \left\{r, q, q+x, s, s+x \right\})$, and if $k > 0$, then $r \in  \text{int}(\conv \left\{r+x, q, q+x, s, s+x \right\})$.
\end{proof}
We can immediately see that the two points in $\mathrm{bd}(H_1)\cap \mathrm{bd}(H_2)$, described in Lemma \ref{PL}, are singular points in the boundary of $H_1\cap H_2$. If they weren't singular then the Gauss image of one intersection point, $x\in \mathrm{bd}(H_1)\cap \mathrm{bd}(H_2)$, with respect to the two homothets, will be the same and a singleton unit vector. Since these homothets are strictly convex their Gauss images will be injective, and hence the inverse images of the point $x$ from the two homothets to $C$ will be the same. This implies that the homothets are subsets of each other a contradiction to the fact their intersection was assume to be proper.

Turning onto the implications Lemma \ref{PL} has to $C$-polygons, first we can immediately see that any $C$-polygon will have a finite amount of vertices. As a vertex is a boundary point that is in the boundary of two generating homothets, and so the set of vertices is a subset of the set of all boundary intersections between each pair of homothets. Which if we have $n$ generating homothets we have at most $2\binom{n}{2}$ many boundary intersections between all the pairs of generating homothets implying $|Vert(H)|\leq 2\binom{n}{2}$. Since we are in the plane, this also implies the number of edges a $C$-polygon has is finite.

We can also notice that if we have a boundary point $x\in \mathrm{bd}(H)$, then it must be in the boundary of some generating homothet $H_j$. This means we can classify boundary points of $H$ into two cases, the boundary points that are in the boundary of only one generating homothet, and boundary points that are in the boundaries of multiple generating homothets. We will soon see that the latter will be vertices and the former will be elements in the relative interior of edges. The next two lemmas establish that vertices are indeed singular points and edges are smooth boundary arcs of $H$ between vertices.

\begin{lemma} \label{Lvert}
Let $H$ be a $C$-polygon where $C$ is a smooth strictly convex domain, then a point $v\in \mathrm{bd}(H)$ is a vertex if and only if it's a singular point of $H$.
\end{lemma}
\begin{proof}
Suppose $v\in \mathrm{bd}(H)$ is a vertex, this means it is in the boundary intersection of two generating homothets call them $H_i$ and $H_j$. By construction we have that $H\subseteq H_i\cap H_j$ and that $v$ is a boundary element of both sets. From this we obtain that $\Gamma_{H_i\cap H_j}(v)\subseteq \Gamma_{H}(v)$, as supporting  lines of $H_i\cap H_j$ will support subsets of $H_i\cap H_j$ that they intersect with. We established $\Gamma_{H_i\cap H_j}(v)$ is not a singleton point and hence $\Gamma_{H}(v)$ will also not be a singleton, showing that vertices are singular points of $H$.

Conversely suppose $v\in \mathrm{bd}(H)$ is a singular point, then as it's a boundary point it must be either a boundary element of one generating homothet or multiple. If it's the latter we are done, so suppose this singular boundary point was only in the boundary of one generating homothet $H_j$, and in the interior of all the other generating homothets. Then the boundary points of $H$ arbitrarily close to this point must also be only contained in the boundary of $H_j$, implying $H_j$ is not smooth, a contradiction.

\end{proof}

\begin{lemma} \label{Ledge}
Edges are smooth boundary arcs between vertices and vice versa.
\end{lemma}
\begin{proof}
For the forward implication consider an edge in an edge family, this is a non-singleton maximally connected boundary arc of some homothet that is contain in $H$. Since this arc is in the boundary of one homothet it must be in the boundary of $H$, this arc is also smooth as $C$ is smooth and it's end point must be vertices otherwise it would not be maximally connected. 
For the reverse implication consider an arbitrary smooth boundary arc between two vertices $v_1,v_2\in Vert(H)$, and an arbitrary point, $x$, in this arc. It is a smooth boundary point by assumption which implies $x$ must only be contained in the boundary of one homothet, $H_j$. Notice that since the arc is smooth, if one point in this smooth arc is in the boundary of only one homothet, the entire smooth arc must be as well. If there was a transition to another generating homothet along this arc, we must have a transition point. This is because boundaries are closed, and so their intersection must have an overlapping point, and that point would be a vertex and thus singular. So the entire arc belongs to only one homothet, this arc is also contain in the boundary of $H$ and is maximally connected and is not a singleton point which means it's an edge in the edge family $E_j$.

\end{proof}

\section{Proof of Theorem \ref{T1}, the case when \textit{m = 0}}\label{S3}

There are two statements to prove in Theorem \ref{T1}, we will first prove the $C$-polygon statement and then the translative $C$-polygon statement. Please note throughout this entire section $C$ will be a smooth strictly convex domain.
\subsection{The \textit{C}-polygon case.}\label{SS3.1}
Let $H=\cap_{i=1}^n H_i$ be a $C$-polygon, we need to show that $Vert(H)$ has cardinality between $n$ and $2(n-1)$. The lower bound is easily handled with an upcoming observation, so we will focus on the upper bound.
We will prove this bound by induction on $n$, the number of generating homothets. We start with the base case where $n=2$, we aim to show that $|Vert(H)|\leq 2=2(n-1)$. We have already seen in Lemma \ref{PL}, that $|Vert(H)|=|Vert(H_1\cap H_2)|=|\mathrm{bd}(H_1)\cap \mathrm{bd}(H_2)|=2$ completing the base case.

Our inductive hypothesis is that any $C$-polygon made up of $n$ generating homothets will have at most $2(n-1)$ vertices. We aim to show that any $C$-polygon made up of $n+1$ generating homothets will have at most $2n$ vertices. To that end let $H=\cap_{i=1}^{n+1} H_i$ be a $C$-polygon made up of $n+1$ generating homothets. 

The object we will be applying our inductive hypothesis on is denoted and defined as $W_j=\bigcap_{i=1,i\neq j}^{n+1}H_i$. The notation $W_j$ is meant to invoke that $H$ is without the $j$'th homothet in the intersection creating it. This is in fact a $C$-polygon as $n\geq 2$, it has non-empty interior as it's a super-set of $H$ and is a reduced intersection by the following lemma.
\begin{lemma} \label{L5}
If $U_i\subseteq \mathbb{E}^d$ for all $i\in \{1,2,\dots, n\}$, and $\cap_{i=1}^n U_i$ is reduced and $n\geq 2$, then for all $j\in \{1,2,\dots,n\}$, $\cap_{i=1,i\neq j}^nU_i$ is reduced. 
\end{lemma}
\begin{proof}

Let $U = U_1 \cap U_2 \cap \dots \cap U_n$, and $S \subseteq \left\{1, \dots, n\right\}$. We denote $$W_S = \bigcap_{i\in \left\{1, \dots, n\right\} \setminus S} U_i,$$ as defined before, $W_{\{j\}}=W_j$. Note that for any $i \in  \{1,2, \dots, n\}$ we have $U = W_i \cap U_i$.

By assumption, $U$ is reduced, which is equivalent to $U \neq W_i$ for any $i \in \{1,2, \dots, n\}$. To show that $W_i$ is also reduced for any $i \in \{1,2, \dots, n\}$, we need to demonstrate that for any distinct $i,j \in \{1,2, \dots, n\} $, we have $W_i \neq W_{\left\{i,j\right\}}$. Indeed if $W_i= W_{\left\{i,j\right\}}$, we would have $U = W_i \cap U_i  = W_{\left\{i,j\right\}} \cap U_i = W_j$. Thus, we contradict the initial assumption that $U$ is reduced.

\end{proof}

An important property of $W_j$ is, if $H_j$ is a generating homothet of a $C$-polygon $H$, then 

\begin{equation}\label{eq1}
\mathrm{bd}(H_j)\cap \mathrm{int}(W_j)\neq \emptyset
\end{equation}

This can easily be proven by contradiction due to the fact $H$ is a proper intersection. Another important observation is that the amount of boundary arcs of $H_j$ that cross the interior of $W_j$ is the amount of edges in the edge family $E_j$. With this observation we can immediately prove the lower bound in Theorem \ref{T1} as each edge family has cardinality at least 1 by equation (\ref{eq1}). Notice that any new vertices found in $H$ that are not in $W_j$ must be in the relative boundary of the edges in the edge family $E_j$. 

\begin{figure}[!ht] \label{F2}
\centering
\begin{subfigure}{.45\textwidth}
  \centering
  \begin{tikzpicture} 
\node [above right, inner sep=0] (image) at (0,0) {\includegraphics[width=0.8\linewidth]{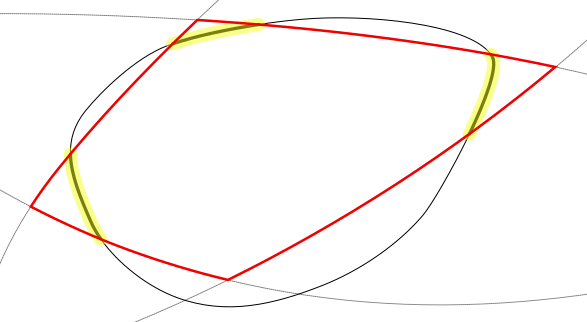}};    
\begin{scope}[
x={($0.1*(image.south east)$)},
y={($0.1*(image.north west)$)}];
        
    \node[left] at (0.5,4){\large \textcolor{red}{$W_j$}};
    \node[left] at (8.5,3){\large $H_j$}; 
       
\end{scope}
\end{tikzpicture}
\end{subfigure}%
\begin{subfigure}{.5\textwidth}
  \centering
    \begin{tikzpicture} 
\node [above right, inner sep=0] (image) at (0,0) {\includegraphics[width=0.9\linewidth]{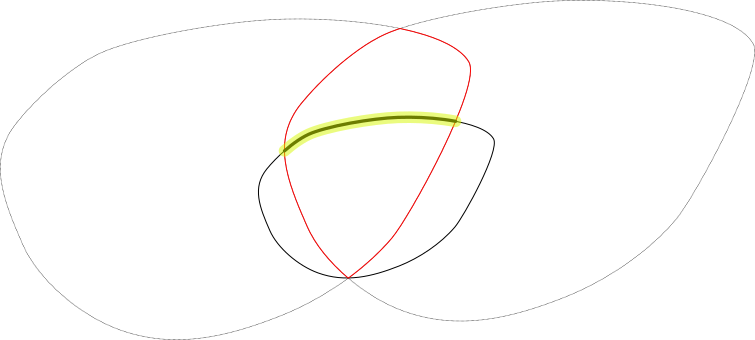}};    
\begin{scope}[
x={($0.1*(image.south east)$)},
y={($0.1*(image.north west)$)}];
       
    \node[above left] at (4,7){\textcolor{red}{$W_j$}};
    \node[left] at (7,3){\large $H_j$}; 
    
\end{scope}
\end{tikzpicture}
\end{subfigure}
`   \caption{Two examples of how $H_j$ can intersect $W_j$.}
\end{figure}
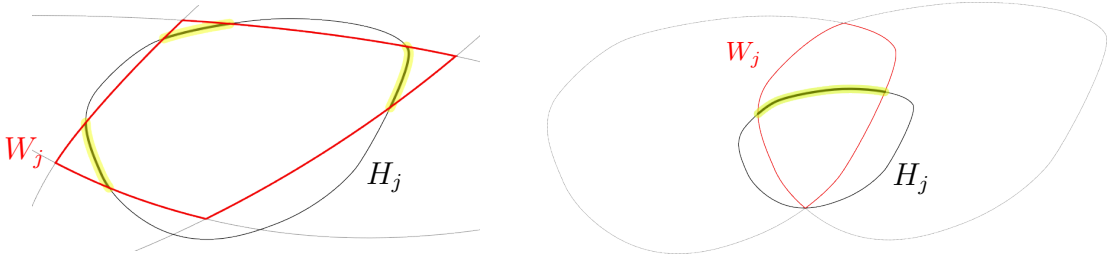
To complete the inductive step we will show that every $C$-polygon will contain a singleton-edge family. This would complete the induction because if we have this, lets call the singleton-edge family $E_k$, then the $\mathrm{bd}(H_k)$ intersects the interior of $W_k$ in one and only one boundary arc. This means that the new vertices added on to $W_k$ by the inclusion of $H_k$ in the intersection to form $H$ will be at most two, with this we obtain:

$$|Vert(H)|\leq |Vert(W_j)|+|Vert(H)\setminus Vert(W_j)|\leq 2(n-1)+2=2n $$

So all we need to show is that every $C$-polygon contains a singleton-edge family. We do this by utilizing a notion of gaps which are defined as follows.
\begin{definition} \label{D5}
Let $H_j$ be a generating homothet of $H$, then consider $\mathrm{bd}(H_j)\setminus  \mathcal{E}_j$ which is a union of $|E_j|$ many maximally connected open boundary arcs of $H_j$. Let us denote these open boundary arcs of our generating homothet by $g_1,g_2,\dots, g_{|E_j|}$. We will call $\conv(\cl(g_k))$ a gap of $H_j$ for $k\in \{1,2,\dots, |E_j|\}$ and we denote and call $G_j=\cup_{i=1}^{|E_j|}\{conv(\cl (g_i))\}$ the gap family of $H_j$.
\end{definition}
\begin{figure}[!ht] \label{F3}
\centering
\begin{subfigure}{.33\textwidth}
  \centering
  \begin{tikzpicture} 
\node [above right, inner sep=0] (image) at (0,0)   {\includegraphics[width=0.8\linewidth]{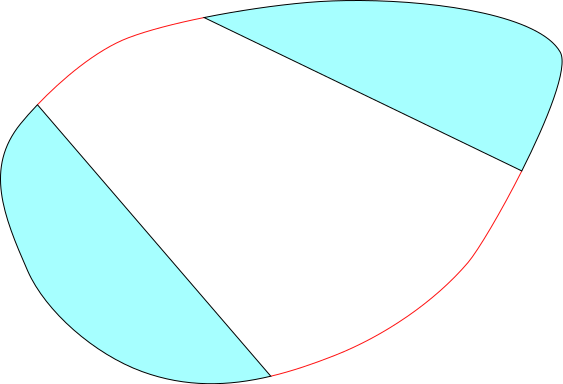}};  
\begin{scope}[
x={($0.1*(image.south east)$)},
y={($0.1*(image.north west)$)}];
        
    
\end{scope}
\end{tikzpicture}
\end{subfigure}%
\begin{subfigure}{.33\textwidth}
  \centering
    \begin{tikzpicture} 
\node [above right, inner sep=0] (image) at (0,0) {\includegraphics[width=0.8\linewidth]{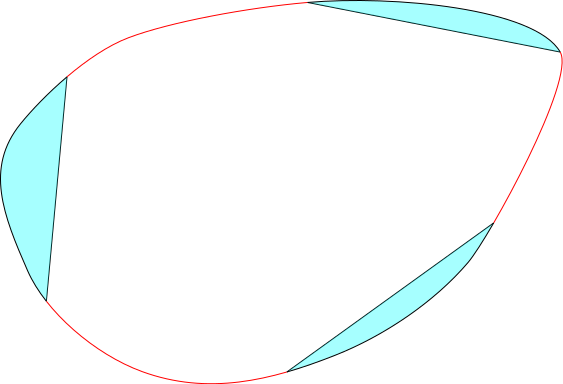}};   
\begin{scope}[
x={($0.1*(image.south east)$)},
y={($0.1*(image.north west)$)}];
       
    
\end{scope}
\end{tikzpicture}
\end{subfigure}
\begin{subfigure}{.33\textwidth}
  \centering
    \begin{tikzpicture} 
\node [above right, inner sep=0] (image) at (0,0) {\includegraphics[width=0.8\linewidth]{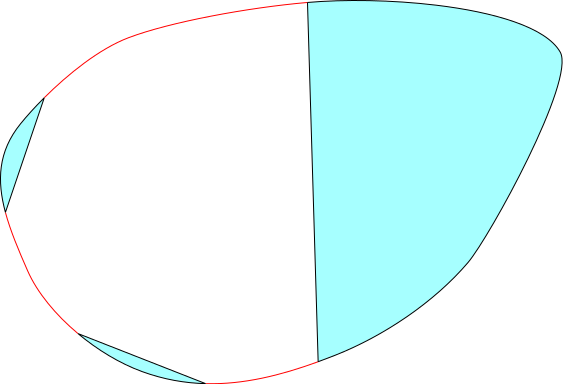}};     
\begin{scope}[
x={($0.1*(image.south east)$)},
y={($0.1*(image.north west)$)}];
       
    
\end{scope}
\end{tikzpicture}
\end{subfigure}
\caption{Three examples of gap families of a homothet.}
\end{figure}
\begin{lemma} \label{L6}
For any $j\in \{1,2,\dots n+1\}$:
$$(\mathrm{bd}(H)\setminus E_j)\subseteq G_j $$
\end{lemma}
\begin{proof}
What this lemma says is that the boundary arcs of $H$ connecting the edges generated by a particular homothet, must lie in the gaps of that homothet. This can be proven by observing that the line segment connecting the two vertices bounding a gap must be in the interior of $H$ because $H$ is strictly convex, and the boundary of $H$ cannot go outside of the boundary of $H_j$ since $H\subseteq H_j$.

\end{proof}
\begin{lemma}\label{L7}
If $H_i$ and $H_j$ are two distinct generating homothets of $H$, then the intersection of $\mathrm{bd}(H_i)$ with the gap family of $H_j$ is non-empty and is contained in a single gap of $H_j$.
\end{lemma}

\begin{proof}
$H_i$ must contain the edge family $E_j$, and so when the $\mathrm{bd}(H_i)$ intersects a gap, its boundary must enter and exit the gap from the boundary of $H_j$. The homothets cannot be tangent to each other as that would contradict Lemma \ref{PL}, so when the $\mathrm{bd}(H_i)$ intersects a gap it pierces the interior of it, and the $\mathrm{bd}(H_i)$ must intersect at least one gap in order for the intersection to be reduced. So every gap $\mathrm{bd}(H_i)$ intersects generates two intersection points in $\mathrm{bd}(H_i)\cap \mathrm{bd}(H_j)$, but of course Lemma $\ref{PL}$ implies this is at most two proving the lemma.

\begin{figure}[!ht] \label{F4}
\centering

  \begin{tikzpicture} 
\node [above right, inner sep=0] (image) at (0,0) 
{\includegraphics[scale=0.2]{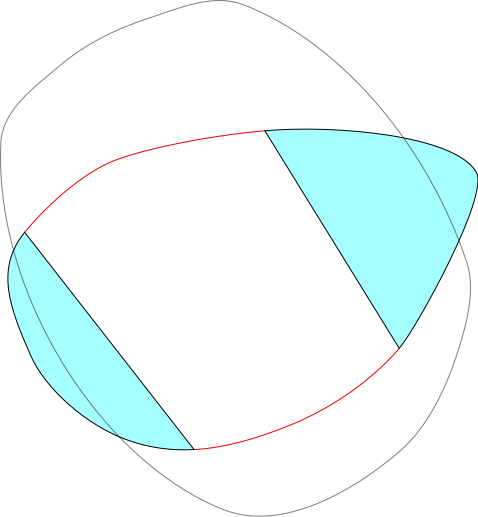}};     
\begin{scope}[
x={($0.1*(image.south east)$)},
y={($0.1*(image.north west)$)}];
        
    
\end{scope}
\end{tikzpicture}

\caption{Example of what would happen if a homothet intersected multiple gaps.}
\end{figure}
\end{proof}
As a result of the previous 2 lemmas we see the following.
\begin{corollary} \label{C8}
No two edges in the same edge family may lie in different gaps of a gap family.
\end{corollary}

It is with this final corollary that we are able to show that every $C$-polygon will have a singleton-edge family. First as $n+1\geq 3$ we have at least three edge families in $H$. Consider an arbitrary edge family $E_k$, if this is a singleton-edge family we are done so suppose it's a multi-edge family, consider an arbitrary gap of the homothet $H_k$. This gap must have at least one edge in it as the boundary of $H$ connecting the two vertices bounding a gap, is a closed connected curve made out of edges. These edges connecting the vertices must also be contained in the gap of $H_k$ according to Lemma \ref{L6}. 

Take the edge adjacent to one of the vertices bounding the gap, that is also contained in the gap of $H_k$ and let $E_l$ be the edge family of this edge. If this is a singleton-edge family we are done, if it's a multi-edge family then the edges of this edge family must be contained in the gap of $H_k$ by Corollary \ref{C8}. So there is a gap of $H_l$ that is contained in the gap of $H_k$ which we can repeat this process with. Since we have a finite amount of edges this process must end and so we must obtain a singleton-edge family eventually.
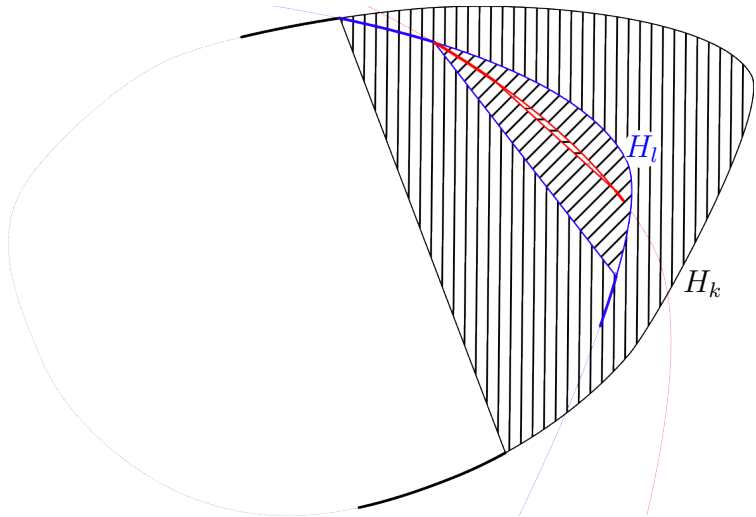
\begin{figure}[!ht]\label{F5}
\centering
  \begin{tikzpicture} 
\node [above right, inner sep=0] (image) at (0,0) {\includegraphics[width=0.6\linewidth]{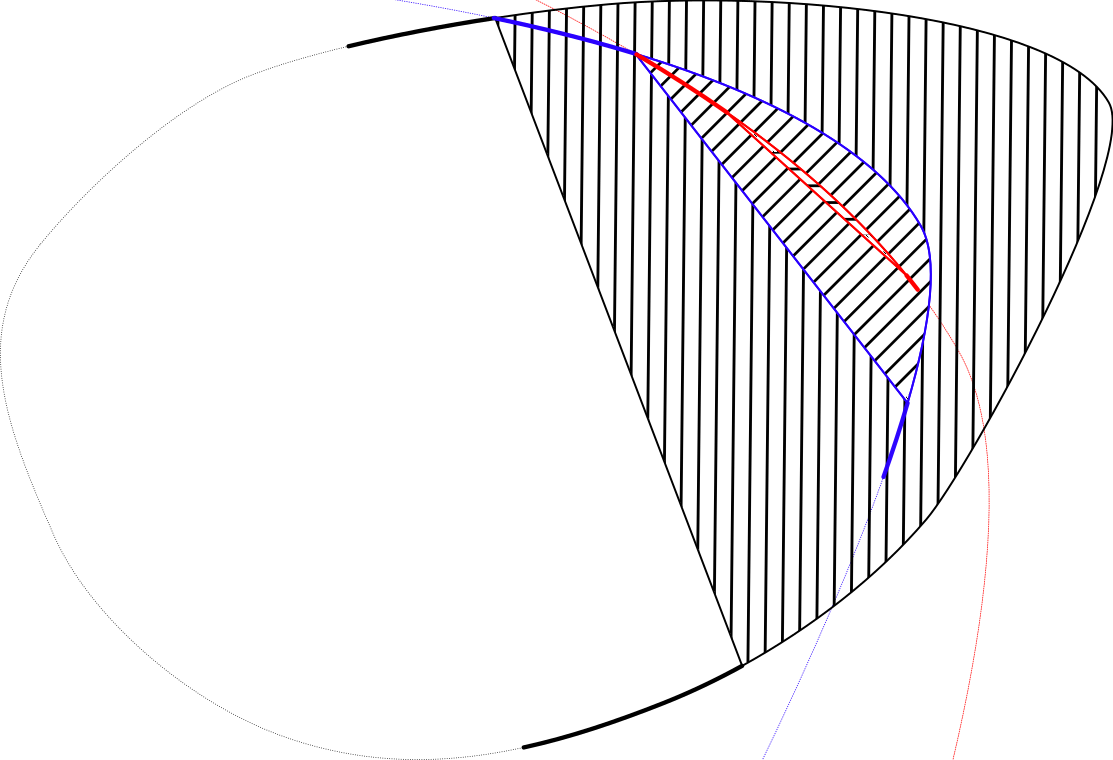}};    
\begin{scope}[
x={($0.1*(image.south east)$)},
y={($0.1*(image.north west)$)}];
        
    \filldraw [white] (8.5,7.2) circle (7pt);    
    \node[below left] at (9.7,5){\large $H_k$};  
    \node[] at (8.5,7.2){\large \textcolor{blue}{$H_l$}};

\end{scope}
\end{tikzpicture}
\caption{Depiction of the process of how to find a singleton-edge family.}
\end{figure}

The following theorem is the most general statement this section has proven for the case where $m=0$. Notice we didn't use the fact that the generating homothets of our $C$-polygon are homothets of $C$ besides in deriving Lemma \ref{PL}. So if we assume our intersection satisfies Lemma \ref{PL}, and remove the restriction that the domains we are intersecting are homothets, our proof will remain identical. 

\begin{theorem} \label{T9}
If $X=\cap_{i=1}^nX_i$ where $X_i$ is a smooth strictly convex domain for each $i\in \{1,2,\dots, n\}$, where $X$ is a proper intersection, and for every distinct $i,j\in \{1,2,\dots,n\}$, $|\mathrm{bd}(X_i)\cap \mathrm{bd}(X_j)|=2$, then the cardinality of singular points on $X$ is between $n$ and $2(n-1)$.
\end{theorem}

\subsection{The translative \textit{C}-polygon case:}
First it should be noted that all the above results established in section \ref{S2} and \ref{SS3.1} apply to translative $C$-polygons as well, this implies that a translative $C$-polygon has at least as many edges as generating translates, since every edge family is non-empty. This implies the desired lower bound for translative $C$-polygons.

For the upper bound we prove it by induction in an analogous way to the $C$-polygon case done previously, as a result we omit some details. Our base case is also handled by Lemma \ref{PL}, so let $T=\cap_{i=1}^{n+1}T_i$ be a translative $C$-polygon where $C$ is a smooth strictly convex domain. Our inductive hypothesis is that any translative $C$-polygon made up of $n$ generating translates will have at most $n$ vertices, we wish to show $T$ has at most $n+1$ vertices. We already know there must be a singleton-edge family by section \ref{SS3.1}, let $E_j$ be that singleton-edge family. In order to complete the proof we only need to show that the inclusion of $T_j$ to $W_j$, to form $T$, net total increases the vertex count by at most one. We have seen that $T_j$ is an edge family that increases our count of vertices by at most two in section \ref{SS3.1}. However, in our translative case we will show we must exclude at least one of the vertices of $W_j$ and hence net total can only increase our vertex count by at most one. We do this with the help of the following lemma.

\begin{lemma}\label{L10}
Let $T_1$ and $T_2$ be translates of $C$ that produce a proper intersection, then the boundary of $T_1$ in the exterior of $T_2$ has a Gauss image that contains a hemisphere.
\end{lemma}
\begin{proof}
\begin{figure}[!ht]
\centering

  \begin{tikzpicture} 
\node [above right, inner sep=0] (image) at (0,0) {\includegraphics[width=0.6\linewidth]{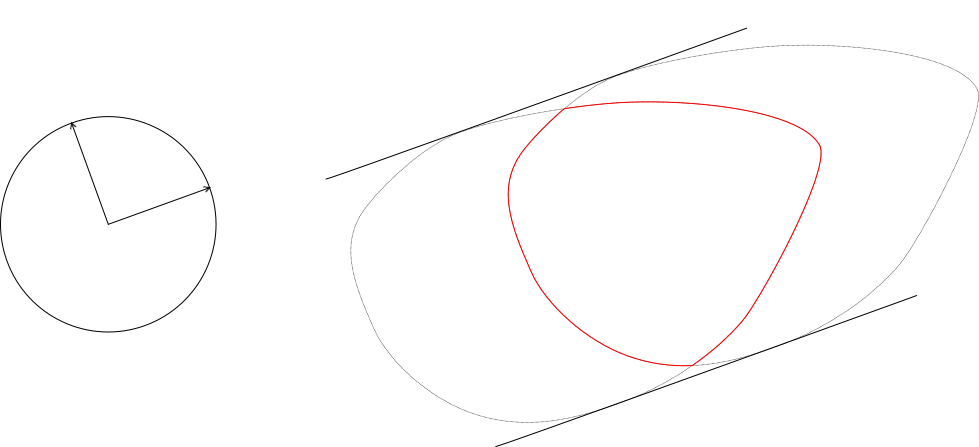}};    
\begin{scope}[
x={($0.1*(image.south east)$)},
y={($0.1*(image.north west)$)}];
    \node[] at (0,3){\large $\mathbb{S}^1$}; 
    \node[] at (1.6,5){\large $\tau$};  
    \node[] at (0.6,6){\large $\rho$};  
    \node[] at (5,9){\large $H_{\rho}$};  
    \node[] at (7,0.8){\large $H_{-\rho}$};   
\end{scope}
\end{tikzpicture}
\caption{Intersection between two translates of $C$.}
\label{F6}
\end{figure}
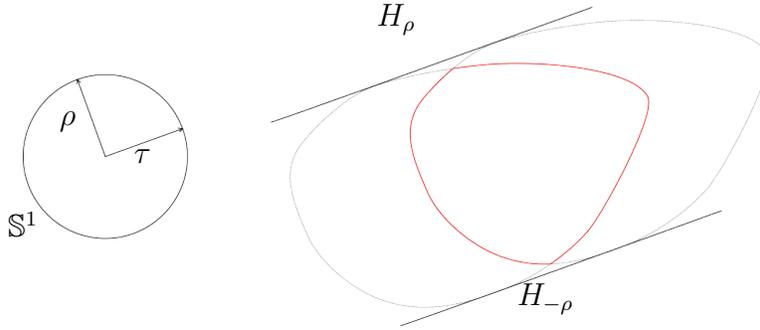

First let $\tau\in \mathbb{S}^1$ be the direction of the translation to get from $T_1$ to $T_2$, and let $\pm \rho\in \mathbb{S}^1$ be the two perpendicular vectors of $\tau$. Since $C$ has a bijective function for its Gauss mapping, $T_1$ and $T_2$ will each have two unique boundary points with Gauss image $\pm \rho$. Notice by the choice of $\rho$ the support lines will be equal, and both translates will be contained in the band between these two support lines, which is depicted in Figure \ref{F6}. Each open boundary arc of $T_1$ and $T_2$ between the two points of contact with the support lines,  will have a Gauss image of an open hemisphere on $\mathbb{S}^1$ centered at $\pm \tau$. Lemma \ref{PL} implies only two of these arcs have a boundary intersection with the other translate proving the lemma.

\end{proof}

This lemma also implies that the relative interior of any edge has a Gauss image that is contained in an open hemisphere. Since $T_j$ has one only edge in its edge family, it has only one maximally connected boundary arc that intersects the interior of $W_j$. This boundary arc will enter and exit the boundary of $W_j$ in two places and the only way it can include all vertices of $W_j$ is if the boundary of $T_j$ enters and leaves the boundary of $W_j$ in the relative interior of same edge.

This would contradict Lemma \ref{L10} as the relative interior of edges have a Gauss image that is contained in a hemisphere, but the pairwise intersection between $T_j$ and the translate that generates the edge of $W_j$ that $\textrm{bd}(T_j)$ intersects, violates Lemma \ref{L10}. This proves the induction as we know $W_j$ has $n$ vertices by our inductive hypothesis and the inclusion of $T_j$ to $W_j$ increases our vertex count by at most one implying $|Vert(T)|=|Vert(W_j\cap T_j)|\leq n+1$.

\begin{figure}[!ht]
    \centering    
\begin{tikzpicture}
\node [above right, inner sep=0] (image) at (0,0) {\includegraphics[scale=0.3]{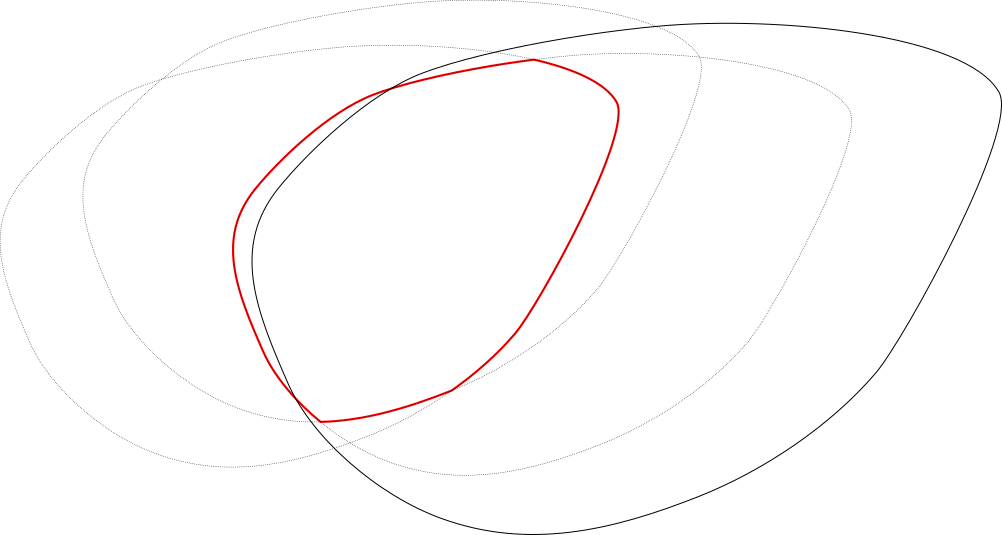}};

   \begin{scope}[
x={($0.1*(image.south east)$)},
y={($0.1*(image.north west)$)}];
       
    \node[left] at (2.5,7){\textcolor{red}{$W_j$}};
    \node[left] at (9.5,2.5){\large $H_j$}; 
    
\end{scope}

\end{tikzpicture}
    \caption{Diagram depicting that in order to contain all vertices of $W_j$ we need a homothetic copy of $C$ and not a translative one.}
\end{figure}

\section{Proof of Theorem \ref{T1}: the positive \textit{m} case } \label{S4}

When we consider a $C$-polygon, $H=\cap_{i=1}^{n} H_i$, where $C$ has $m$ singular points, things behave similarly to the smooth case. Notice Lemma \ref{PL} holds for strictly convex domains so we know that the boundary intersection between pairs of homothets will have a cardinality of two, and these two points are singular in the intersection body. We define edge families identically although now edges in edge families may not be smooth as boundary arcs of $C$ may not be smooth. This observation lets us classify vertices of $H$ into one of two kinds. Pairwise vertices are singular points of $H$ that are in the boundary of two generating homothets. Inherited vertices are singular points that are in the relative interior of an edge. Notice during the proofs in section \ref{S3} our arguments did not use the fact that our edges were smooth. This implies that we have at most $2(n-1)$ many pairwise vertices in the $C$-polygon case, and $n$ many pairwise vertices in the translative $C$-polygon case. Notice that each homothet has a non-empty edge family implying our lower bound for both the $C$-polygon and translative $C$-polygon cases. 

Edges in an edge family will be non-singleton maximally connected components of a generating homothet. We can view these edges as maximally connected components of the boundary of $C$, by inverting the homothet map. The inherited singular points on an edge will corresponds with a singular point of $C$ since the inverted homothetic image of $C$ preserve the Gauss image. Since $H$ is strictly convex it's Gauss mapping is injective which implies when we represent the edges of $H$ on $C$, they will not overlap. This implies $H$ has at most $m$ inherited vertices completing the proof.

\section{Sharpness of Results}\label{S5} 
In this section we discuss the sharpness of the bounds established in Theorem \ref{T1}. Given a class of convex domains $\mathcal{C}$, and a bound on the number of vertices a $C$-polygon has when $C\in \mathcal{C}$, we say the bound is strongly sharp when for all $C\in \mathcal{C}$ and every natural $n\geq 2$ there is a $C$-polygon with $n$ generating homothets that realizes the bound. If there exists $C\in \mathcal{C}$ such that for every $n\geq 2$ we can find a $C$-polygon with $n$ generating homothets that can realize the bound, we call the bound weakly sharp.

\begin{claim}
The upper bound of $2(n-1)+m$ vertices for $C$-polygons where $C$ is a strictly convex domain is strongly sharp.
\end{claim}
\begin{proof}
Given any strictly convex domain, $C$, with $m$ singular points, we must have a smooth strictly convex boundary arc of $C$. Pick $n-1$ many points on the relative interior of this boundary arc, by expanding $C$ with a homothet centre at each of these boundary points we can create $n-1$ homothets to be arbitrarily close to the supporting lines of the $n-1$ boundary points of $C$. Then we move these $n-1$ homothets arbitrarily inward to pierce the interior of $C$ and create $2$ pairwise vertices with each homothet giving us $2(n-1)+m$ vertices. Figure \ref{F8} depicts the construction.
\begin{figure}[!ht]
    \centering
    
\begin{tikzpicture}
\node [above right, inner sep=0] (image) at (0,0)
{\includegraphics[scale=0.3]{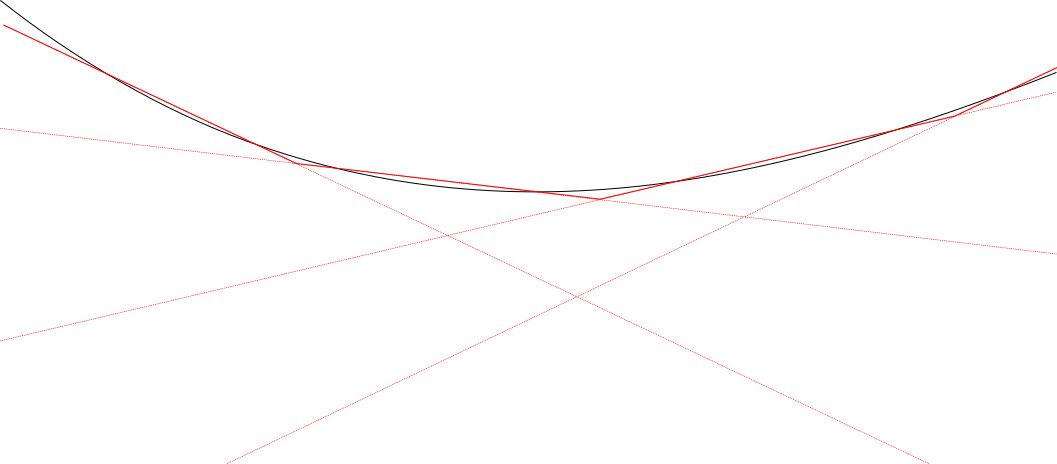}};
   \begin{scope}[
x={($0.1*(image.south east)$)},
y={($0.1*(image.north west)$)}];   
\end{scope}

\end{tikzpicture}
    \caption{Sharpness of upper bound construction for $C$-polygons where $C$ is a strictly convex domain.}
    \label{F8}
\end{figure}

\end{proof}

\begin{claim}
The lower bound of $n$ vertices for $C$-polygons and translative $C$-polygons, where $C$ is a strictly convex domain, is weakly sharp. However, the lower bound of $n$ vertices for $C$-polygons and translative $C$-polygons is strongly sharp when $C$ is a smooth strictly convex domain.
\end{claim}
\begin{proof}
The latter part of the claim is trivial as any translative $C$-polygon with $n$ generating translates, where $C$ is a smooth strictly convex domain, will have exactly $n$ vertices according to Theorem \ref{T1}. For the former part of the claim, the previous sentence shows that the bound is at least weakly sharp. So we need only show this bound is not strongly sharp which means we need to show there exists a strictly convex domain $C$ and a natural $n\geq 2$ where every $C$-polygon with $n$ generating homothets has more than $n$ vertices. We need only choose $n=2$ and $C$ as depicted in Figure \ref{F9}. We construct $C$ by intersecting three circles, this is a ball-polygon and we choose it such that the antipodal point of the Gauss image of any boundary point of $C$ that is smooth, is contained in the Gauss image of a vertex of $C$. By Lemma \ref{PL} we know any intersection of two homothets of $C$ will have two pairwise vertices. So we need to show every $C$-polygon generated by two homothets of $C$ will have at least one inherited vertex. The following statement completes the proof and is left to the reader. Given an intersection of two homothets, $H_1\cap H_2$, of a strictly convex domain, there exists a point in the relative interior of each of the two edges, call them $x_1\in \mathrm{bd}(H_1)$ and $x_2\in \mathrm{bd}(H_2)$, such that $\Gamma_{H_1}(x_1)$ contains a point that is antipodal to a point in $\Gamma_{H_2}(x_2)$.
\begin{figure}[!ht]
    \centering    
\begin{tikzpicture}
\node [above right, inner sep=0] (image) at (0,0) {\includegraphics[scale=0.2]{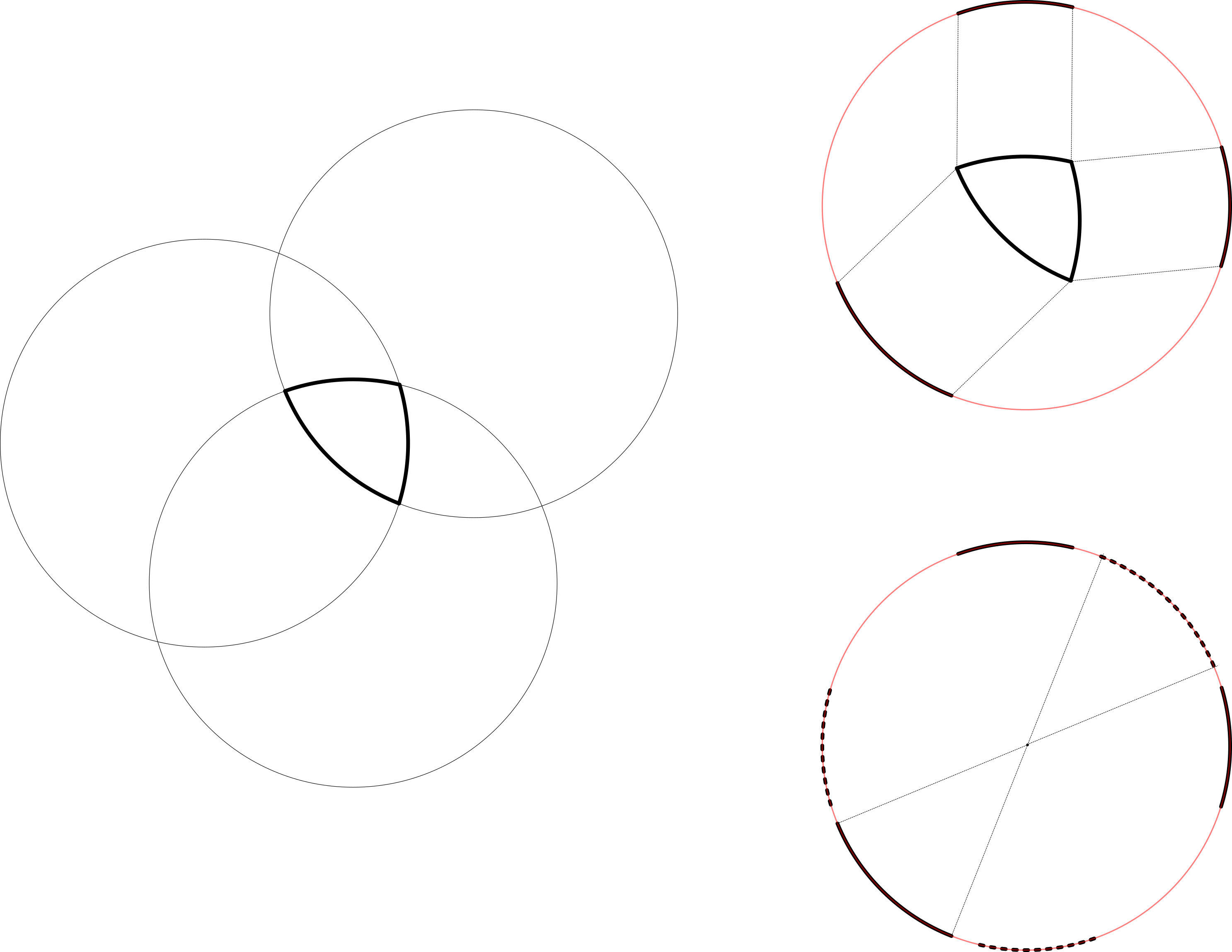}};

   \begin{scope}[
x={($0.1*(image.south east)$)},
y={($0.1*(image.north west)$)}]; 
\end{scope}

\end{tikzpicture}
    \caption{Construction of $C$.}
    \label{F9}
\end{figure}

\end{proof}
\begin{claim}
For every natural $n\geq 2$ there exist a convex domain, $C$, and a $C$-polygon, $H,$ with $n$ generating homothets such that $H$ has zero vertices.
\end{claim}
\begin{proof}

First notice that this condition on the bound of zero vertices is weaker than weak sharpness. Given an arbitrary natural $n\geq 3$ we choose our convex domain $C$ to be a polygon with $n$ vertices that have sufficiently small rounded corners. To construct a $C$-polygon with $n$ generating homothets we take $C$ and enlarge the other $n-1$ copies of $C$ to smoothly transition with $n-1$ many smoothed corners of $C$ as depicted in Figure \ref{F10} for the case where $n=3$ and $4$. For the case where $n=2$ two translates of the rounded square can construct a translative $C$-polygon with zero vertices.

\begin{figure}[!ht]
\centering
\begin{subfigure}{.49\textwidth}
  \centering
    \begin{tikzpicture} 
\node [above right, inner sep=0] (image) at (0,0) {\includegraphics[width=0.8\linewidth]{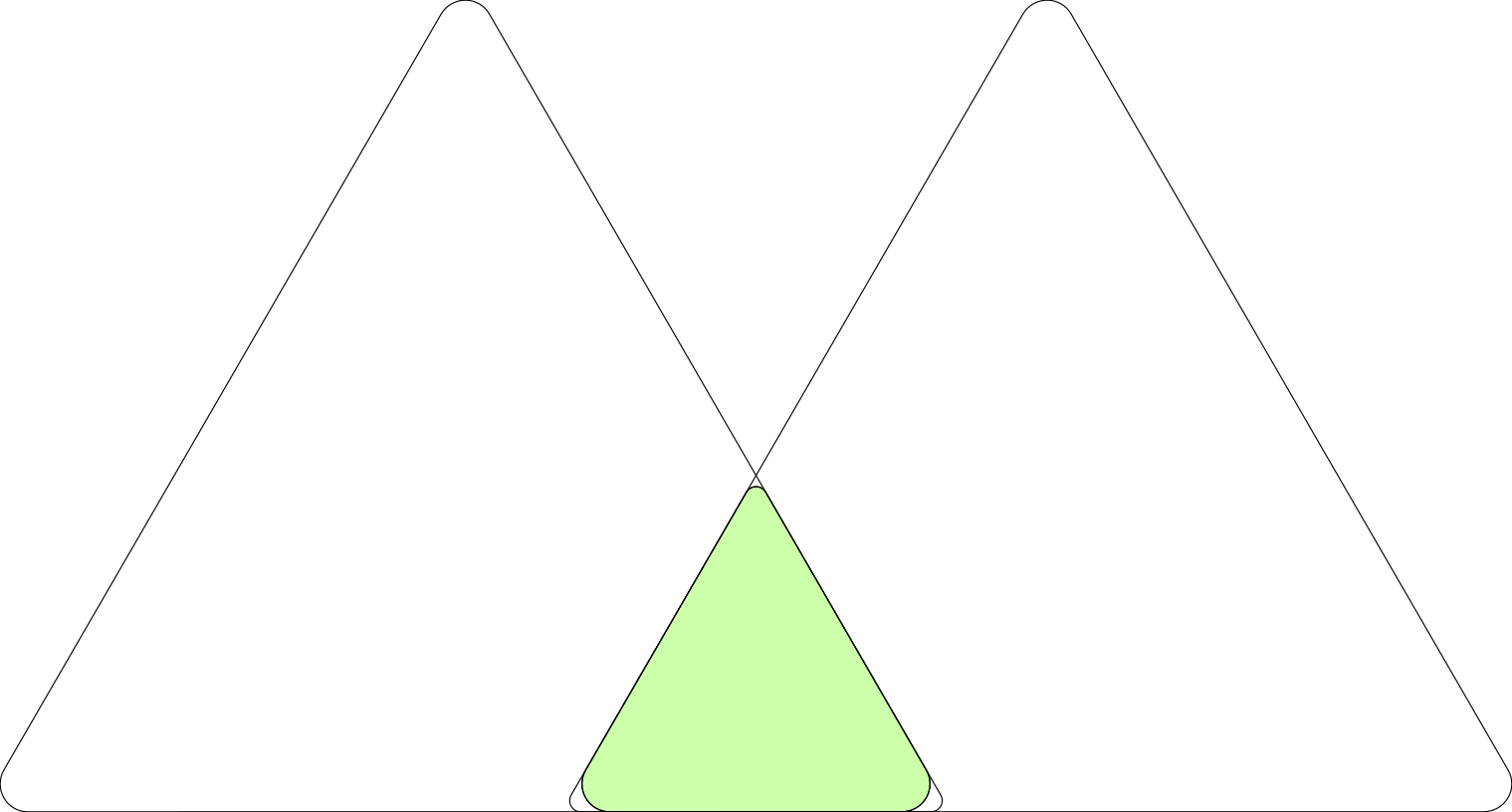}};  
\begin{scope}[
x={($0.1*(image.south east)$)},
y={($0.1*(image.north west)$)}];
       
    
\end{scope}
\end{tikzpicture}
\end{subfigure}
\begin{subfigure}{.49\textwidth}
  \centering
    \begin{tikzpicture} 
\node [above right, inner sep=0] (image) at (0,0) 
{\includegraphics[width=0.5\linewidth]{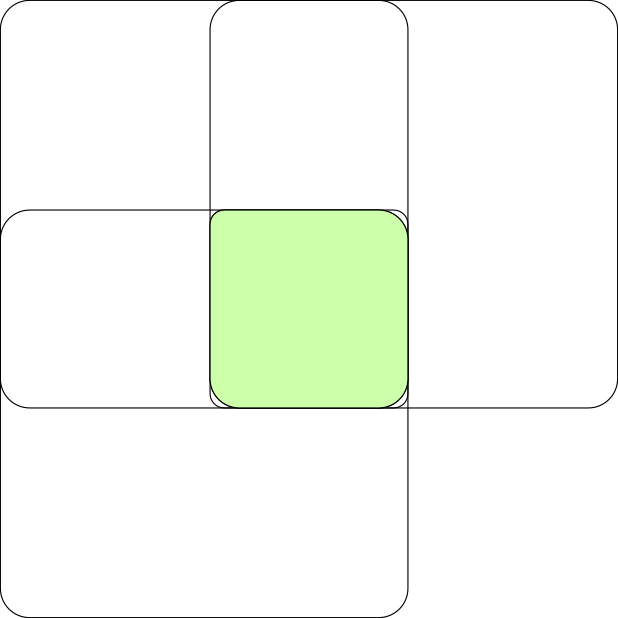}};
\begin{scope}[
x={($0.1*(image.south east)$)},
y={($0.1*(image.north west)$)}];
       
    
\end{scope}
\end{tikzpicture}
\end{subfigure}
\caption{Depiction of the construction of a smooth $C$-polygon.}
\label{F10}
\end{figure}
\end{proof}

We end this paper with an open problem. 
\begin{problem}
If $C$ is a convex domain with $m$ singular points, then the number of singular points a $C$-polygon with $n$ generating domains will have is at most $2(n-1)+m$.
\end{problem}

\bigskip


\noindent Illya Ivanov \\
\small{Department of Mathematics and Statistics, University of Calgary, Canada}\\
\small{E-mail: \texttt{ilya.ivanov1@ucalgary.ca}}

\bigskip

\noindent Cameron Strachan \\
\small{Department of Mathematics and Statistics, University of Calgary, Canada}\\
\small{E-mail: \texttt{braden.strachan@ucalgary.ca}}

\end{document}